\documentclass[12pt,tbtags,leqno]{amsart}

\usepackage{cite}

\usepackage{graphicx}
\usepackage{yfonts}

\usepackage[T1]{fontenc}
\usepackage[latin1]{inputenc}
\usepackage[english]{babel}
\usepackage{amssymb}
\usepackage{amsthm}
\usepackage{amsmath}
\usepackage{verbatim}
\usepackage{color}

\numberwithin{equation}{section}

\newcommand{\MH}{\mathrm{Mult}(\mathcal{H})}

\newcommand{\HH}{\mathcal{H}}
\newcommand{\HK}{\mathcal{K}}

\newcommand{\HE}{\mathcal{E}}

\newcommand{\HB}{\mathcal{B}}

\newcommand{\dB}{\partial \mathbb B_d}

\newcommand{\D}{\mathbb{D}}
\newcommand{\B}{\mathbb{B}}
\newcommand{\C}{\mathbb{C}}
\newcommand{\N}{\mathbb{N}}

\newcommand{\la}{\langle}
\newcommand{\ra}{\rangle}
\newcommand{\Bd}{\mathbb B_d}

\newcommand{\Hol}{\operatorname{Hol}}
\newcommand{\Mult}{\mathrm{Mult}}

\def\HE{{\mathcal E}}

\def\om{\omega}

\def\Bd{{\mathbb{B}_d}}

\newcommand{\Ker}[1]{\mathsf{Ker}~}

\def\Section2{ Section 2}
\def\Lemmsixseven{Lemma 6.7 }
\def\Lemmfourtwo{Lemma 4.2}
\def\Theooneeight{Theorem 1.8 }
\def\Theoonesix{Theorem 1.6}
\def\onethree{1.3 }
\def\twooneandtwofive{2.1 and 2.5 }

\DeclareMathOperator{\mim}{Im}

\theoremstyle{plain}
\newtheorem{theorem}{Theorem}[section]
\newtheorem{lemma}[theorem]{Lemma}

\newtheorem*{theo}{Theorem}

\newtheorem{ques}[theorem]{Question}
\theoremstyle{definition}

\begin{document}

\date{\today}

\title[Cyclicity and iterated logarithms]{Cyclicity and iterated logarithms in the Drury-Arveson space}
\author[A. Aleman]{Alexandru Aleman}
\address{Lund University, Mathematics, Faculty of Science, P.O. Box 118, S-221 00 Lund, Sweden}
\email{alexandru.aleman@math.lu.se}

\author[K.-M. Perfekt]{Karl-Mikael Perfekt}
\address{Department of Mathematical Sciences, Norwegian University of Science and Technology (NTNU), 7491 Trondheim, Norway}
\email{karl-mikael.perfekt@ntnu.no}

\author[S. Richter]{Stefan Richter}
\address{Department of Mathematics, University of Tennessee, 1403 Circle Drive, Knoxville, TN 37996-1320, USA}
\email{srichter@utk.edu}

\author[C. Sundberg]{Carl Sundberg}
\address{Department of Mathematics, University of Tennessee, 1403 Circle Drive, Knoxville, TN 37996-1320, USA}
\email{csundber@utk.edu}

\author[J. Sunkes]{James Sunkes}
\address{Huntsville, AL, USA}
\email{jsunkes@gmail.com}

\maketitle
\begin{abstract}
Let $H^2_d$ be the Drury-Arveson space, and let $f\in H^2_d$ have bounded argument and no zeros in $\Bd$. We  show that $f$ is cyclic in $H^2_d$ if and only if $\log f$ belongs to the Pick-Smirnov class $N^+(H^2_d)$.  Furthermore, for non-vanishing functions $f\in H^2_d$ with bounded argument and $H^\infty$-norm less than 1, cyclicity can also be tested via iterated logarithms. For example, we show that $f$ is cyclic if and only if $\log(1+\log (1/f))\in N^+(H^2_d)$. Thus, a sufficient condition for cyclicity is that $\log(1+\log (1/f))\in H^2_d$. More generally, our results hold for all radially weighted Besov spaces that also are complete Pick spaces.
\end{abstract}

\section{Introduction}
For $d\in \N$ let $\Bd$ denote the open unit ball of $\C^d$. Let $\HH$ be a Hilbert space of functions on $\Bd$. We write $\Mult(\HH)=\{\varphi\in \Hol(\Bd): \varphi f\in \HH \text{ for all }f\in \HH\}$ for the multipliers of $\HH$, and we say that $f\in \HH$ is cyclic in $\HH$, if $\{\varphi f: \varphi\in \MH\}$ is dense in $\HH$.

The results of this article apply to radially weighted Besov space $\HH$ that are also complete Pick spaces; see Section 2 for precise definitions. Perhaps the most interesting examples of spaces that fall in both categories are the Dirichlet space of the unit disc $\D=\mathbb{B}_1$
$$D=\{f\in \Hol(\D): \int_\D |f'|^2 dA<\infty\}$$ and the Drury-Arveson space $H^2_d$, which for $d=1$ coincides with the classical Hardy space $H^2(\D)$.
  $H^2_d$ is defined to be the unique Hilbert  space of analytic functions on $\Bd$  that  has reproducing kernel $k_w(z)=\frac{1}{1-\la z,w\ra}$. Alternatively, $H^2_d$ can be defined by saying that an analytic function $f(z)=\sum_{\alpha\in \N_0^d} \hat{f}(\alpha)z^\alpha\in H^2_d$ if
 $$\|f\|^2_{H^2_d} = \sum_{\alpha\in \N_0^d} \frac{\alpha !}{|\alpha|!}|\hat{f}(\alpha)|^2<\infty.$$
As usual, if $\alpha=(\alpha_1,\dots,\alpha_d)\in \N_0^d$ and $z=(z_1,\dots, z_d)\in \C^d$, then $|\alpha|=\alpha_1+\dots +\alpha_d$, $\alpha!=\alpha_1!\cdot \dots \cdot \alpha_d!$, and $z^\alpha=z_1^{\alpha_1}\cdot \dots \cdot z_d^{\alpha_d}$.

The Drury-Arveson space has turned out to be important for multivariable operator theory. Many classical results that are known to be true and useful for the Hardy space $H^2(\D)$ have an extension to $H^2_d$ and many other complete Pick spaces.  A fact about  all complete Pick spaces $\HH$ that will be used repeatedly in this paper is that $\HH$ is contained in the corresponding Pick-Smirnov class
$$N^+(\HH)=\{\varphi/\psi : \varphi, \psi \in \Mult(\HH), \psi \text{ cyclic}\},$$ see \cite{AHMRsccps}.

Now assume that $\HH$ is a radially weighted Besov space that  also is a complete Pick space. In that case using the fact that $\HH\subseteq N^+(\HH)$ it is an elementary observation that $f\in \HH$ is cyclic, if and only if $1/f\in N^+(\HH)$, see Theorem \ref{1overf}. Thus, using that $N^+(\HH)$ is an algebra we see that a sufficient condition for cyclicity of $f$ is $f^{-1/n}\in \HH$ for some positive integer $n$. That makes it reasonable to consider $\log f$.  If $f(z) \ne 0$ for all $z\in \Bd$, then we say that $f$ has bounded argument, if there is $M>0$ such that $| \mim \log f|\le M$ in $\Bd$ for some branch of $\log f$. Of course, for different branches of $\log f$ the condition will still hold, but possibly with a different $M$. For $n\in \N$ define analytic functions $G_n$ on the right half plane by $G_1(z)=z$ and $G_{n+1}(z)=\log(1+G_n(z))$, $n \in \N$.
\begin{theorem}\label{thm1} Let $\HH$ be a radially weighted Besov space that  also is a complete Pick space. If $f\in \HH$ is zero-free on $\Bd$ and has bounded argument, then

(a)  $\log f \in N^+(\HH)$ if and only if $f$ is cyclic in $\HH$,

(b) if also $\|f\|_\infty \le 1$, then the following are equivalent:
\begin{enumerate}
\item $f$ is cyclic in $\HH$,
\item there is $n \in \N$ such that $G_n\circ \log(1/f)\in N^+(\HH)$,
\item for every $n \in \N$,  $G_n\circ \log(1/f)\in N^+(\HH).$
\end{enumerate}
Hence, in case (b), a sufficient condition for cyclicity is that $G_n\circ \log(1/f)\in \HH$ for some $n\in \N$.
\end{theorem}

We say that a $d$-variable polynomial is stable if it does not have any zeros in $\Bd$. Stable polynomials have bounded argument, see Lemma~6.1. For the Dirichlet space $D$ it is known that all stable polynomials are cyclic \cite{BrownShields}. In Lemma \ref{lem:LogPolyDiri} we will show that $\log(1+\log\frac{2^np(0)}{p(z)})\in D$ for all stable  polynomials of degree $\le n$, that is, the cyclicity can be deduced from the sufficient condition of Theorem \ref{thm1} with $n=2$.

If $d\ge 4$, then it is known that there are stable polynomials that are not cyclic in $H^2_d$ (\cite{CyclicWeighedBesov_I}, {\Section2}). Hence there are nonzero functions $f\in H^2_d$ with bounded argument such that $G_n\circ \log(1/f)\notin N^+(H^2_d)$ for any $n\in \N$. However, if $p$ is a stable polynomial
 that only depends on the variables $z_1$ and $z_2$, then it is known that $p$ is cyclic in $H^2_d$, see \cite{KosinskiVavitsas} and \cite{CyclicWeighedBesov_I}. As in the case of the Dirichlet space, in Section \ref{polys} we will show that the result can also be established by use of the sufficient condition of Theorem \ref{thm1} with $n=2$. Furthermore, in Section \ref{polys} we will  give additional examples of cyclic polynomials $p$ in $H^2_d$ with zeros on $\dB$ that satisfy  $\log (1+\log\frac{c}{p})\in H^2_d$ for an appropriate constant $c$.

\begin{ques} If $f\in H^2_d$ is cyclic in $H^2_d$, has bounded argument, and $\|f\|_\infty \le 1$, then is there $n$ such that $G_n\circ \log(1/f)\in H^2_d$? If so, is $n$ independent of $f$?
\end{ques}

If $f$ is not bounded below, then $G_n\circ \log (1/f)$ will be unbounded, but its growth near $\dB$ will be slower the larger $n$ is. Thus, in a sense, for functions with bounded argument, Theorem \ref{thm1} establishes a connection between cyclicity and the exceptional sets of elements in the class $N^+(\HH)$.  If $\HH=H^2(\D)$ and $f$ is a singular inner function, then $N^+(H^2(\D))$ is the classical Smirnov class, and $f$ is not cyclic in $H^2(\D)$ even though $\log f\in N^+(H^2(\D))$.  This shows that if one simply deletes the condition on the argument of $f$ in Theorem \ref{thm1}, then the theorem would no longer be true. On the other hand, if $f$ is the conformal map from the unit disc onto a cornucopia that is bounded away from its center 0, then $f$ is an outer function (hence cyclic in $H^2(\D)$) that does not have bounded argument. Thus, it would be interesting to know, whether the hypothesis of having bounded argument can be weakened in Theorem \ref{thm1}.
For the Dirichlet space $D$,  in a forthcoming paper, we will establish the sufficiency of the logarithmic condition under a hypothesis where the boundedness of the argument is replaced by $f$ being outer.

This paper is organized as follows. In Section 2 we have collected definitions and known facts that will be important for the proof of Theorem \ref{thm1}. Section 3 contains some elementary facts about functions in complete Pick spaces, while in Section 4 we present crucial lemmas that hold in radially weighted Besov spaces. Theorem \ref{thm1} is proved in Section 5. In Section 6 we verify that certain stable polynomials satisfy the hypothesis of Theorem \ref{thm1}, and Section 7 contains an open question along with a partial result.


\section{Preliminaries}
\subsection{Definitions}
 A Hilbert function space $\HH$ on a set $X$ is a Hilbert space of complex valued functions on $X$ such that for each $z\in X$ the evaluation functional $f\to f(z)$ is continuous on $\HH$. If $\HH$ and $\HK$ are Hilbert function spaces, then we write $$\Mult(\HH,\HK)=\{\varphi: \varphi f\in \HK \text{ for all }f\in \HH\}$$ for the multipliers from $\HH$ to $\HK$. As above we write $\Mult(\HH)=\Mult(\HH,\HH)$. Each multiplier $\varphi\in \Mult(\HH,\HK)$ defines a bounded multiplication operator $M_\varphi: \HH\to \HK, f\to \varphi f$, and $\Mult(\HH,\HK)$ becomes a Banach space by setting $\|\varphi\|_{\Mult(\HH,\HK)}=\|M_\varphi\|_{\HB(\HH,\HK)}$.

Each Hilbert function space $\HH$ has a reproducing kernel $k: X \times X\to \C$. Writing $k_w(z)=k(z,w)$ it  satisfies $f(w)=\la f, k_w\ra$ for all $f\in \HH$, $w\in X$.  A  reproducing kernel $k$ on $X$ is called a normalized complete Pick kernel, if there is $w_0\in X$ and a function $b$ from $X$ into some auxiliary Hilbert space $\HK$ such that $b(w_0)=0$ and
$$k_w(z)= \frac{1}{1-\la b(z),b(w)\ra_\HK}.$$ In the interesting case where $\HH$ is a Hilbert space of analytic functions one easily shows that $\HH$ is separable, and then one may assume that $\HK$ is separable also. As mentioned in the Introduction the Dirichlet space $D$ of the unit disc and the Drury-Arveson space $H^2_d$ are examples of spaces with a normalized complete Pick kernel. Spaces with normalized complete Pick kernel have many of the properties that one may be familiar with from the Hardy space $H^2(\D)$. For this paper the most important property is that $\HH\subseteq N^+(\HH)$, see \cite{AHMRsccps} and also see Theorem \ref{SarasonAnalog} below for further details. Note that  this property is preserved, if the norm on $\HH$ is replaced with an equivalent norm. Thus, we define a complete Pick space $\HH$ to be a Hilbert function space such that there is an equivalent norm on $\HH$ with the property that the reproducing kernel for the equivalent norm is a normalized complete Pick kernel.

Let $d\in \N$. We will use $\Hol(\Bd)$ to denote the analytic functions on $\Bd$. A radially weighted Besov space is a space of the form $$B^N_\om=\{f\in \Hol(\Bd): R^Nf\in L^2(\om)\}.$$ Here $N$ is a non-negative integer, $R=\sum_{j=1}^d z_j \frac{\partial}{\partial z_j}$ is the radial derivative operator, and $\om$ is a measure on $\overline{\Bd}$ of the type $d\om(z)=d\mu(r) d\sigma(w)$, where $z=rw$, $\sigma$ is the normalized rotationally invariant measure on $\dB$, and $\mu$ is a Borel measure on $[0,1]$ with $\mu((r,1])>0$ for each real $r$ with  $0<r<1$. Such measures will be called admissible radial measures. If $\mu$ has a point mass at 1, then the $L^2(\om)$-norm of an analytic function $f$ is to be understood by $\|f\|^2_{L^2(\om)}=\int_{\Bd}|f|^2d\om + \mu(\{1\})\|f\|^2_{H^2(\dB)}$. We define a norm on $B^N_\om$ by
\begin{align}\label{NormBN} \|f\|^2_{B^N_\om}= \left\{\begin{matrix}& \|f\|^2_{L^2(\om)}, &\text{ if }N=0,\\& \om(\Bd)|f(0)|^2+\|R^N f\|^2_{L^2(\om)}, &\text{ if } N>0,\end{matrix}\right.\end{align}
 and we note that the hypothesis on $\mu$ implies that  each $B^N_\om$ is a Hilbert function space on $\Bd$. For later reference we also note that
 \begin{align}\label{NormNandN-1}\|f\|^2_{B^N_\om}=\om(\Bd)|f(0)|^2+\|Rf\|^2_{B^{N-1}_\om}\end{align} holds for all $N>0$.

The Drury-Arveson space $H^2_d$ is a radially weighted Besov space.
 One calculates $\|f_n\|^2_{H^2(\dB)}=\frac{n! (d-1)!}{(n+d-1)!}\|f_n\|^2_{H^2_d}$, whenever $f_n$ is a homogeneous polynomial of degree $n$, see e.g. \cite{RiSunkesHankel}, Section 2. By use of that identity one verifies that for all $N\ge 0$
\begin{align}\label{sumNorm}\|f\|^2_{B^N_\om}=\om(\Bd) |f(0)|^2 + \sum_{n=1}^\infty n^{2N}\om_n\|f_n\|^2_{H^2_d},\end{align}  whenever $f=\sum_{n=0}^\infty f_n$ is the representation of $f$ as a sum of homogeneous polynomials of degree $n$. Here $\om_n= \frac{n! (d-1)!}{(n+d-1)!}\int_{[0,1]}r^{2n}d\mu(r)$. It follows that $H^2_d=B^N_\om$ with equivalence of norms for some $N$ and $\omega$, which depend on $d$.  In fact,
\begin{align}\label{DruryArveson} H^2_d=\left\{\begin{matrix} B^{(d-1)/2}_\om \ \text{ if } d \text{ is odd, } \om=\sigma\\B^{d/2}_\om  \ \text{ if } d \text{ is even, } \om=V.\end{matrix}\right.\end{align}
Here $V$ denotes normalized Lebesgue measure on $\Bd$.

More generally, in \cite{AHMRRadiallyWeightedBesov} it was shown that if  $\alpha >-1$ and if $d\om(z) = w(z) dV(z)$ where $\frac{w(z)}{(1-|z|^2)^\alpha}$  is non-decreasing as $|z|\to 1$,
then for $N\ge \frac{\alpha+d}{2}$  the space $B_\omega^N$ is a complete Pick space.

\

If $\HH$ is a Hilbert function space and if $f\in \HH$, then we use $[f]$ to denote the smallest multiplier invariant subspace that contains $f$, $[f]=\overline{\{\varphi f: \varphi \in \MH\}}^{\HH}$. Thus, $f$ is  cyclic, if $[f]=\HH$. In all cases that we will consider in this paper the multipliers of $\HH$ are densely contained in $\HH$. Then $f$ is cyclic in $\HH$, if and only if $1\in [f]$.

\

\subsection{Known theorems that will be used} In this section we have collected some known results that will be used in the proof of our main Theorem.
The  first required result about complete Pick spaces spaces is Theorem 1.1 (i) of \cite{AHMRFactor}, restated for the case where $g_w(z)=1$ and $\HE=\C$.
  \begin{theorem} \label{SarasonAnalog} Let $\HH$ be a Hilbert function space on $X$ such that the reproducing kernel $k$ is a normalized Pick-kernel with $k_{w_0}=1$. For  $f : X\to \C,$  the following are equivalent:
  \begin{enumerate}
  \item $f\in \HH$  and $ \|f\|\le 1 $
\item there are multipliers $\varphi, \psi\in \text{\rm Mult}(\HH)$ such that \begin{enumerate}
\item $f= \frac{\varphi}{1-\psi}$
\item $\psi(w_0)=0$, and
\item $\|\psi h\|^2+\|\varphi h\|^2 \le \|h\|^2 $ for every $h\in \HH$.
\end{enumerate}
\end{enumerate}
\end{theorem}
Part (a) in the next Lemma is Lemma 2.3 of \cite{AHMRsccps} and (b) is \Lemmsixseven of \cite{CyclicWeighedBesov_I}.
\begin{lemma}\label{multIndep} (a) If $\psi\in \MH$ is not constant and $\|\psi\|_{\MH}\le 1$, then $1-\psi$ is cyclic in $\HH$.

(b) If $\HH$ is a complete Pick space, and if $f=\frac{u}{v}\in \HH$ for $u,v \in \MH$, $v$ cyclic, then $[f]=[u]$.
\end{lemma}

The following is Lemma 3.3 of \cite{AHMRUniformSmirnov}.
\begin{lemma}\label{UniformSmirnovLemma} Let $\HH$ be a Hilbert function space on a set $X$ with a normalized complete Pick kernel.

If $f,g\in \MH$ such that $g\in [f]$, then there is a sequence $\{\varphi_j\}\subseteq \MH$ such that $\varphi_j(z)f(z) \to g(z)$ for all $z\in X$ and $\|\varphi_j f\|_{\MH}\le \|g\|_{\MH}$ for all $j\in \N$.
\end{lemma}
In particular it follows that if a multiplier $f\in \Mult(\HK)$ is cyclic in a complete Pick space $\HK$, then there is a sequence $\{\varphi_j\}\subseteq \Mult(\HK)$ such that $\varphi_j(z)f(z) \to 1$ for all $z\in X$, and $\|\varphi_j f\|_\infty \le 1$. Indeed, as we pass from a Hilbert function space $\HH$ with normalized complete Pick kernel to a complete Pick space $\HK=\HH$ we have to keep in mind that equivalent norms on the spaces imply equivalent multiplier norms. But the new statement is true, because $\|u\|_\infty\le \|u\|_{\Mult{\HK}}$ and $\|1\|_{\Mult(\HK)}=1$ for all Hilbert function spaces with $1\in \HK$.
 The following Lemma is known to hold for all radially weighted Besov spaces, and it was also used repeatedly in \cite{CyclicWeighedBesov_I}.
\begin{lemma}[\cite{CyclicWeighedBesov_I}, \Lemmfourtwo]\label{basicMultIncl} Suppose $N\in \N$ and  $\om$ is an admissible radial measure. Let $\varphi$ be an analytic function on $\Bd$.

Then for each integer $k$ with $1\le k\le N$ we have $\varphi \in \Mult(B^k_\om)$ if and only if $\varphi \in \Mult(B^{k-1}_\om)$ and $R\varphi \in \Mult(B^k_\om, B^{k-1}_\om)$.

Furthermore,  $$\|\varphi\|_{\Mult(B^k_\om)}\approx \|\varphi\|_{\Mult(B^{k-1}_\om)}+ \|R\varphi\|_{\Mult(B^k_\om,B^{k-1}_\om)}.$$
\end{lemma}

Finally, we recall \Theooneeight of  \cite{CyclicWeighedBesov_I}.
\begin{theorem} \label{Paper1:AbsoluteValuePickIntro} Let $N\in\N$, and let $B^N_\om$ be a radially weighted Besov space that is also a complete Pick space, and let $f,g\in B^N_\om$.

If $|f(z)|\le |g(z)|$ for all $z\in \Bd$, and if $f$ is cyclic, then $g$ is cyclic.
\end{theorem}


\section{A condition for cyclicity}

 We start with an elementary, yet important observation.
\begin{theorem} \label{1overf} Suppose $\HH$ has a normalized complete Pick kernel. Let $f\in \HH$. Then $f$ is cyclic if and only if $1/f \in N^+(\HH)$.
\end{theorem}
\begin{proof} We assume $\|f\|=1$. Then by Theorem \ref{SarasonAnalog} and Lemma \ref{multIndep}  we have $f=u/(1-v)$ for contractive multipliers $u,v$ with $[f]=[u]$ and $[1-v]=\HH$. Thus, if $f$ is cyclic, then $u$ is cyclic and $1/f=(1-v)/u\in N^+(\HH)$. Conversely, if $1/f\in N^+(\HH)$, then $1/f=a/b$ for multipliers $a$ and $b$ where $b$ is cyclic. Then $au=b(1-v)$. Since $b$ and $1-v$ are cyclic multipliers it is clear that $b(1-v)$ is cyclic, hence $1\in [b(1-v)]=[au] \subseteq [u]=[f]$. Hence $f$ is cyclic.
\end{proof}
Thus, one can prove that a function is cyclic by trying to verify that for some $n$ one has $f^{-1/n}\in \HH$. For $n=1$ and $\HH=H^2_d$ this was one of the results of \cite{RiSunkesHankel}. However,  often this will be false for all $n$ even though $f$ is cyclic. It turns out one  has a better chance by considering $\log f$. In light of Theorem \ref{SarasonAnalog} the following lemma implies that one may restrict attention to multipliers.

\begin{lemma} \label{log(1-phi)} Let $\HH$ be a Hilbert function space on $\Bd$. If $\varphi\in \Mult(\HH)$ with $\|\varphi\|_{\Mult(\HH)}\le 1$ and $\varphi\ne 1$, then $\log(1-\varphi)\in N^+(\HH)$.
\end{lemma}
 \begin{proof} Note that for each $z\in \Bd$ we have $|\varphi(z)|<1$ (Lemma 2.2 of \cite{AHMRsccps}) hence
 \begin{align*}u(z)&=(1-\varphi(z))\log(1-\varphi(z))\\&=-(1-\varphi(z))\sum_{n=1}^\infty \frac{\varphi^n(z)}{n}\\
 &=-\varphi(z)\left(1-\sum_{n=2}^\infty\frac{\varphi^n(z)}{n(n+1)}\right).\end{align*} The last sum converges in the multiplier norm, hence $u$ is a multiplier and $\log(1-\varphi)=\frac{u}{1-\varphi}$ is a ratio of multipliers and we have already mentioned that $1-\varphi$ is cyclic (Lemma \ref{multIndep}).
 \end{proof}


\section{Lemmas about weighted Besov spaces}

Part (a) of the following elementary lemma will be used  in induction arguments. In part (b) we have already carried out such an induction.

\begin{lemma}\label{MultiplierInductionStep} Suppose $N\in \N$ and  $\om$ is an admissible radial measure. Fix an integer $k$ with $1\le k\le N$.

(a)  For each integer $n \ge 0$ there is a  $C>0$ such that whenever $\varphi, \psi \in \Mult(B^k_\om)$ and  $\psi^n e^{-\frac{\varphi}{\psi}} \in \Mult(B^{k-1}_{\om})$, then $\psi^{n+2} e^{-\frac{\varphi}{\psi}} \in \Mult(B^k_{\om})$, and
\begin{equation}\label{psinexp}\|\psi^{n+2} e^{-\frac{\varphi}{\psi}}\|_{\Mult(B^k_\om)} \le C\|\psi^n e^{-\frac{\varphi}{\psi}}\|_{\Mult(B^{k-1}_\om)} (\|\varphi\|_{\Mult(B^k_\om)}+\|\psi\|_{\Mult(B^k_\om)})^2.\end{equation}

(b) If $\varphi \in \Mult(B^N_\om)$ is zero free in $\Bd$ and $j$ is a non-negative integer such that $\varphi^j \log \varphi \in H^\infty$, then there is an integer $n \ge j$ such that $\varphi^{n}\log \varphi \in \Mult(B^N_\om)$.
\end{lemma}

\begin{proof} We start by remarking that radially weighted Besov spaces satisfy the multiplier inclusion condition $ \Mult(B^k_\om)\subseteq  \Mult(B^{k-1}_\om)$ with $\|\cdot\|_{\Mult(B^{k-1}_\om)}\le \|\cdot\|_{\Mult(B^{k}_\om)}$, see Lemma \ref{basicMultIncl}.

(a) We will use Lemma \ref{basicMultIncl}. We assume that $\psi^n e^{-\frac{\varphi}{\psi}} \in \Mult(B^{k-1}_\om)$ for two multipliers $\varphi, \psi \in \Mult(B^k_\om)$ and an integer $n \ge 0$. It is clear that $\psi^{n+2} e^{-\frac{\varphi}{\psi}} \in \Mult(B^{k-1}_{\om})$ with norm estimate
$$\|\psi^{n+2} e^{-\frac{\varphi}{\psi}}\|_{\Mult(B^{k-1}_{\om})}\le \|\psi\|_{\Mult(B^{k-1}_\om)}^2 \|\psi^n e^{-\frac{\varphi}{\psi}}\|_{\Mult(B^{k-1}_\om)}.$$ Using the contractive inclusion  $ \Mult(B^k_\om)\subseteq  \Mult(B^{k-1}_\om)$ we see that this last expression is bounded by the right hand side of equation (\ref{psinexp}).

It remains to show that $R(\psi^{n+2} e^{-\frac{\varphi}{\psi}}) \in \Mult(B^k_\om,B^{k-1}_\om)$. A straightforward calculation shows
$$R(\psi^{n+2} e^{-\frac{\varphi}{\psi}})= \psi^n e^{-\frac{\varphi}{\psi}}\left((n+2) (R\psi) \psi -( (R\varphi)\psi - (R\psi) \varphi )\right),$$ and it is clear from the hypotheses that this function multiplies $B^k_\om$ into $B^{k-1}_\om$. In fact, by use of Lemma \ref{basicMultIncl} we get the norm estimate
\begin{multline*}
\|R(\psi^{n+2} e^{-\frac{\varphi}{\psi}})\|_{\Mult(B^k_\om,B^{k-1}_\om)} \\ \le C(n+2)\|\psi^n e^{-\frac{\varphi}{\psi}}\|_{\Mult(B^{k-1}_{\om})} (\|\varphi\|_{\Mult(B^k_\om)}+\|\psi\|_{\Mult(B^k_\om)})^2.\end{multline*}

(b) Note that the hypothesis says that $\varphi^j \log \varphi\in \Mult(B^0_\om)$.

By induction it will be sufficient to show the following:
If $1\le k\le N$ and if $\varphi^n \log \varphi \in \Mult(B^{k-1}_\om)$, then $\varphi^{n+1} \log \varphi \in \Mult(B^k_{\om})$.

We will use Lemma \ref{basicMultIncl} again. Since it is clear from the multiplier inclusion condition that $\varphi^{n+1} \log \varphi \in \Mult(B^{k-1}_\om)$, it is enough to show that
  $R(\varphi^{n+1} \log \varphi)$ multiplies $B^k_\om$ into $B^{k-1}_\om$. We have
$R(\varphi^{n+1} \log \varphi)= ((n+1) \log \varphi +1)\varphi^nR\varphi$. Since $\varphi\in \Mult(B^k_\om)$ $R\varphi$ multiplies $B^k_\om$ into $B^{k-1}_\om$. The result follows since $((n+1) \log \varphi +1)\varphi^n \in \Mult(B^{k-1}_\om)$.
\end{proof}
\begin{theorem} \label{technicalTheo} Let $\HH$ be a radially weighted Besov space.

 There are $C>0$ and $n\in \N$ such that whenever  $\varphi, \psi \in \Mult(\HH)$ with Re $\frac{\varphi}{\psi} \ge 0$, then  $\psi^n e^{-\frac{\varphi}{\psi}} \in \Mult(\HH)$ and $$\|\psi^n e^{-\frac{\varphi}{\psi}} \|_{\Mult(\HH)} \le C (\|\psi\|_{\Mult(\HH)}+\|\varphi\|_{\Mult(\HH)})^n.$$
\end{theorem}
\begin{proof} Let $N\in \N$ and let $\om$ be an admissible radial measure such that $\HH=B^N_\om$. Recall that the contractive inclusions $\Mult(B^{k}_\om)\subseteq \Mult(B^{k-1}_\om)$ hold for each integer $k$ with $1\le k \le N$.

 The hypothesis implies that  $e^{-\frac{\varphi}{\psi}} \in \Mult(B^0_\om)=H^\infty$ with $\|e^{-\frac{\varphi}{\psi}}\|_\infty\le 1$.
Then by Lemma \ref{MultiplierInductionStep} (a) we have
$$\|\psi^2 e^{-\frac{\varphi}{\psi}}\|_{\Mult(B^1_\om)}\le C(\|\varphi\|_{\Mult(B^1_\om)}+\|\psi\|_{\Mult(B^1_\om)})^2$$

Repeated application of Lemma \ref{MultiplierInductionStep} (a) implies that $\psi^{2N} e^{-\frac{\varphi}{\psi}} \in \Mult(B^N_\om)$ with
$$\|\psi^{2N} e^{-\frac{\varphi}{\psi}}\|_{\Mult(B^N_\om)}\le C (\|\varphi\|_{\Mult(B^N_\om)}+\|\psi\|_{\Mult(B^N_\om)})^{2N}.$$ Here, of course, the constant has changed, but it only depends on $N$.
\end{proof}


\section{Logarithms and iterated logarithms} We are now going to work with logarithms and make our way  towards proving Theorem \ref{thm1}.  For later reference we note that if $f$ and $g$ are zero-free analytic functions such that $g$ has bounded argument, then $fg$ has bounded argument, if and only if $f$ has. For technical reasons it will be convenient to make the following definition:
$$N^+_{ba}(\HH)= \{\frac{\varphi}{\psi}: \varphi, \psi \in \Mult(\HH), \psi \text{ cyclic and has bounded argument}\}.$$
Note that Theorem \ref{SarasonAnalog} implies that if $\HH$ is a complete Pick space, then $\HH\subseteq N^+_{ba}(\HH)$.

\begin{theorem} \label{logTheo} Let $\HH$ be a radially weighted Besov space that also is a complete Pick space.

Let $f\in H^\infty$ be zero free and have bounded argument.
If $\log f\in N^+(\HH)$, then $f, 1/f \in N^+(\HH)$.
\end{theorem}

\begin{proof} Without loss of generality we may assume that $\|f\|_\infty \le 1$ and we assume $| \mim \log f|\le M$. Since $\log f \in N^+(\HH)$, we can find multipliers $u,v\in \Mult(\HH)$ such that $v$ is cyclic and $\log \frac{1}{f} =\frac{u}{v}$.  Then note that
 Re $\frac{ u}{v} \ge 0$. Thus we can  use Theorem \ref{technicalTheo}  to choose a positive integer $n$ such that $v^n f\in \Mult(\HH)$. Since $v^n$ is cyclic this implies that $f\in N^+(\HH)$. We will show that $1/f\in N^+(\HH)$ by showing that $v^nf$ is cyclic in $\HH$ (see Theorem \ref{1overf}).

 By Theorem \ref{technicalTheo} we have \begin{align}\label{fbeta} \|v^nf\|_{\Mult(\HH)} \le C(\|u\|_{\Mult(\HH)}+\|v\|_{\Mult(\HH)})^{n}.\end{align}

The cyclicity of $v$ implies that there are multipliers $p_j$ such that $p_jv\to 1$ in $\HH$. By Lemma \ref{UniformSmirnovLemma} and the remark following it there is a constant $c>0$ and there are multipliers $b_j$ such that $b_j(z)v(z) \to 1$ pointwise as $j\to \infty$, $\|b_jv\|_\infty\le 1$, and $\|b_j v\|_{\MH}\le c$ for each $j$. Then $\text{Re}(1- b_j v)\ge 0$ and since $e^{ b_j u}\in \Mult(\HH)$ we have
$$v^nf e^{  b_ju}\in [v^nf].$$
Note that $v^nf e^{ b_ju}= v^n e^{-\frac{ u}{v}(1- b_j v)}$ and
\begin{align*}\text{Re}\left(\frac{ u}{v}(1- b_j v)\right)&=  \left(\text{Re}\frac{ u}{v}\right)\left(\text{Re}(1- b_j v)\right)+ \left(\mim\frac{ u}{v}\right)(\mim ( b_jv))\\ & \ge 0 -  |\mim \log f| \ge -M.\end{align*}
Thus we can apply Theorem \ref{technicalTheo}  with $\varphi= u(1- b_j v) +Mv$ and $\psi=v$ and obtain
 for each $j$
 \begin{align*}\|v^n  e^{-\frac{ u}{v}(1- b_j v)}\|_{\Mult(\HH)}&\le Ce^M(\| u(1- b_j v) +Mv\|_{\Mult(\HH)}+\|v\|_{\Mult(\HH)})^{n}\\
 &\le Ce^M((1+c)\|u\|_{\Mult(\HH)}+(M+1)\|v\|_{\Mult(\HH)})^n.\end{align*} Thus $v^nf e^{ b_ju}\to v^n $ weakly in $\HH$ as $j\to \infty$. This implies $v^n \in [v^nf]$. But $v^n$ is cyclic, hence  $v^nf$ is cyclic and  $1/f\in N^+(\HH)$.
\end{proof}
\begin{lemma} \label{iterate} Let $\HH$ be a radially weighted Besov space  that also is a complete Pick space, and let $G(z)=\log(1+z)$ be defined in the right half plane. Here $\log$ denotes the principal branch.

Let $F\in \Hol(\Bd)$ with Re$F(z) > 0$ for all $z\in \Bd$.

(a) If
$G \circ F \in N^+(\HH)$, then $F \in N^+(\HH)$.

(b) Conversely, if $F \in N^+_{ba}(\HH)$, then $G\circ F\in N^+_{ba}(\HH)$.
\end{lemma}

\begin{proof} (a) Since Re$F(z)\ge 0$ we have $\frac{1}{1+F}\in H^\infty$ and it is clear that $|\mim\log \frac{1}{1+F}|\le \pi/2$, hence Theorem \ref{logTheo} applies and we conclude that $1+F\in N^+(\HH)$. This implies that $F \in N^+(\HH)$.

(b) Suppose $F=\frac{u}{v}$ for multipliers $u,v$ where $v$ is cyclic and has bounded argument. Since $1\le |1+F|=|1+u/v|$ we have $|v| \le |u+v|$ in $\Bd$. By Theorem \ref{Paper1:AbsoluteValuePickIntro} the cyclicity of $v$ implies that $u+v$ is cyclic. Then we note that for each positive integer $n$ we have that $v^n(u+v)^n$ is a cyclic multiplier as well. Clearly $1+F$ has bounded argument, hence the assumption on $v$ implies that $u+v=v(1+F)$ has bounded argument also, and hence so does $v^n(u+v)^n$ for each positive integer $n$. Furthermore, if a multiplier $\varphi$ has no zeros and bounded argument, then $\varphi \log \varphi \in H^\infty$. By Lemma \ref{MultiplierInductionStep}(b) there is an integer $n$ such that $\varphi^n \log \varphi \in \Mult(\HH)$. We apply this with $\varphi=v$ and also with $\varphi=u+v$ and conclude that there is $n$ such that $$v^n(u+v)^n(log(u+v)-log v) \in \Mult(\HH).$$ (b) follows from this.
\end{proof}

We now prove Theorem~\ref{thm1}, which we restate for convenience. Recall that $G_1(z)=z$ and $G_{n+1}(z)=G(G_n(z))$ for $n\ge 1$.
\begin{theo} Let $\HH$ be a radially weighted Besov space that  also is a complete Pick space. If $f\in \HH$ is zero-free on $\Bd$ and has bounded argument, then
	
	(a)  $\log f \in N^+(\HH)$ if and only if $f$ is cyclic in $\HH$,
	
	(b) if also $\|f\|_\infty \le 1$, then the following are equivalent:
	\begin{enumerate}
		\item $f$ is cyclic in $\HH$,
		\item there is $n \in \N$ such that $G_n\circ \log(1/f)\in N^+(\HH)$,
		\item for every $n \in \N$,  $G_n\circ \log(1/f)\in N^+(\HH).$
	\end{enumerate}
\end{theo}

\begin{proof}
(a) It suffices to prove Theorem in the case when $\|f\|\le 1$. Then by Theorem \ref{SarasonAnalog} we know that $f=\frac{\varphi}{1-\psi}$ for contractive multipliers $\varphi, \psi$ with $\psi \ne 1$. In Lemma \ref{log(1-phi)} we had shown that $\log (1-\psi) \in N^+(\HH)$. Furthermore, since Re $1-\psi\ge 0$ it is clear that the argument of $f$ is bounded, if and only if the argument of $\varphi$ is bounded. Thus, since  $[f]=[\varphi]$ it suffices to prove (a) under the additional assumption that $f$ is a multiplier with $\|f\|_{\Mult(\HH)}\le 1$.

In that case $f \in H^\infty$. Thus, if $\log f \in N^+(\HH)$, then the cyclicity of $f$ follows from Theorem \ref{logTheo}.
For the other direction we suppose that $f$ is a cyclic multiplier.
Since $f$ is bounded it follows that $\left|f \log |f|\right|$ is bounded. Thus the hypothesis implies that $f \log f \in H^\infty$. Then  Lemma \ref{MultiplierInductionStep}(b)  implies that $\log f =\frac{w}{f^n}$ for some natural number $n$ and some $w\in \Mult(\HH)$. Since $f^n$ is cyclic and has bounded argument we conclude that $\log f \in N^+_{ba}(\HH) \subseteq N^+(\HH)$.

(b) We assume that $f$ is not constant. Set $F= \log (1/f)$. Since $\|f\|_\infty \le 1$ we have Re $F(z)>0$ in $\Bd$. From the proof of (a) it follows that $f$ is cyclic in $\HH$, if and only if $F\in N^+_{ba}(\HH)$. Thus (b) follows by an iterated application of Lemma \ref{iterate}.
\end{proof}

\section{Polynomials}\label{polys}
Recall that a stable polynomial is a polynomial without zeros in $\Bd$. First we note that stable polynomials always have bounded argument.
\begin{lemma}If $p$ is a  polynomial in $d$ variables of total degree $\le n$, and if $p(z)\ne 0$ for all $z\in \Bd$, then $|\mim \log p(z)- \log p(0)|\le n\pi$.
\end{lemma}
\begin{proof} First let $q(\lambda)$ be a single variable polynomial of degree equal to $ n$ such that $q$ has no zeros in $\D$. Then there are $A_1,...,A_n$ such that $|A_i|\ge 1$ and $$\frac{q'(\lambda)}{q(\lambda)}=\sum_{i=1}^n \frac{1}{\lambda-A_i}.$$ Then for $|z|<1$
$$\log q(z)- \log q(0) =\int_{[0,z]}\frac{q'(\lambda)}{q(\lambda)}d\lambda= \sum_{i=1}^n \int_{[0,z]}\frac{1}{\lambda-A_i}d\lambda.$$ Hence
$$|\mim(\log q(z)- \log q(0))| \le \sum_{i=1}^n | \mim \int_{[0,z]}\frac{1}{\lambda-A_i}d\lambda|\le n\pi.$$ Now if $p$ is a polynomial in $d$ variables of total degree $\le n$, then for each $z\in \Bd$ the function $q_z(\lambda)=p(\lambda z)$ is a single variable polynomial with $\log p(z)=\log q_z(1)$ and $\log p(0)=\log q_z(0)$. Thus the result follows from the one variable case.
\end{proof}

There are some situations where Theorem \ref{thm1} can be used to show that stable polynomials are cyclic. We will use slice functions and start with two lemmas about single variable functions.

\begin{lemma}\label{singleFactor} For $x\ge 1$ let $h(x)= \frac{x^2}{(1+\log x)^2}$.
Then
$$\int_{\D}h\left(\frac{2|\lambda|}{|z-\lambda|}\right)\frac{dA(z)}{\pi}\le 16$$ for all $\lambda\in \C$ with $|\lambda|\ge 1$.
\end{lemma}
\begin{proof}
Note that  $h$ is increasing.
Hence, if $|\lambda|\ge 2$, then for all $z\in \D$ we have
$$h\left(\frac{2|\lambda|}{|z-\lambda|}\right)\le h\left(\frac{2|\lambda|}{|\lambda|-1}\right)\le h(4)\le 16.$$ Thus, the lemma holds at least for $|\lambda|\ge 2$.

If $1\le |\lambda|\le 2$ and $z\in \D$, then
$$h\left(\frac{2|\lambda|}{|z-\lambda|}\right)\le h\left(\frac{4}{|z-\frac{\lambda}{|\lambda|}|}\right)$$ and hence
\begin{align*}\int_{\D} h\left(\frac{2|\lambda|}{|z-\lambda|}\right)    \frac{dA(z)}{\pi}&\le \int_{\D} h\left(\frac{4}{|z-1|}\right) \frac{dA(z)}{\pi}\\
&\le \int_{|w|\le2}h(\frac{4}{|w|}) \frac{dA(w)}{\pi}\\&=\frac{2}{1+\log 2}\le 16. \end{align*}
\end{proof}
Note that if  $p$ is a stable polynomial of degree $ n$, then there are $\lambda_1, \dots, \lambda_n\in \C\setminus \D$ and $c\in \C, c\ne 0$ such that $p(z)=c\prod_{j=1}^n (z-\lambda_j)$. Then $$|p(z)|\le |c| \prod_{j=1}^n (1+|\lambda_j|)\le 2^n|c| \prod_{j=1}^n |\lambda_j|=2^n|p(0)|$$ for all $z\in \D$. Hence, if  $p$ is any stable polynomial of degree $\le  n$, then the polynomial $q(z)=p(z)/(2^np(0))$ satisfies $\|q\|_\infty \le 1$.
\begin{lemma}\label{lem:LogPolyDiri} Let $p$ be a stable polynomial of a single variable and  of degree $\le n$. Set $F(z)=\log(1+\log\frac{2^n p(0)}{p(z)}),$ then $\int_\D|F'(z)|^2 \frac{dA}{\pi} \le 16 n^2 $.
\end{lemma}
\begin{proof} First assume that the degree of $p$ equals $n$. Then there is $c\in \C$ and  $\lambda_1, \dots, \lambda_n\in \C\setminus \D$ such that $p(z)=c \prod_{j=1}^n (z-\lambda_j)$. Then
$$|1+\log\frac{2^n p(0)}{p(z)}|\ge 1+\log\frac{2^n |p(0)|}{|p(z)|}=1+\sum_{j=1}^n \log\frac{2|\lambda_j|}{|z-\lambda_j|} \ge 1+ \log\frac{2|\lambda_{k}|}{|z-\lambda_{k}|}$$ for all $z\in \D$ and $1\le k\le n$. Then
\begin{align*}
|F'(z)|^2&= \frac{1}{|1+\log\frac{2^n p(0)}{p(z)}|^2}\left|\frac{p'(z)}{p(z)}\right|^2\\
&\le \frac{1}{(1+\log\frac{2^n |p(0)|}{|p(z)|})^2}|\sum_{k=1}^n \frac{1}{z-\lambda_k}|^2\\
&\le n \sum_{k=1}^n \frac{1}{(1+\log\frac{2^n |p(0)|}{|p(z)|})^2}\frac{1}{|z-\lambda_k|^2}\\
&\le n  \sum_{k=1}^n \frac{1}{(1+\log\frac{2|\lambda_k|}{|z-\lambda_k|})^2}\frac{4|\lambda_k|^2}{|z-\lambda_k|^2}
\end{align*}
Thus, if the degree of $p$ equals $n$, then the result now follows from Lemma \ref{singleFactor}.

If the degree is $k<n$, then $1+\log \frac{2^n|p(0)|}{|p(z)|}\ge 1+\log \frac{2^k|p(0)|}{|p(z)|}$ and the Lemma follows by taking $n=k$ in the case that we proved above.
\end{proof}
Lemma \ref{lem:LogPolyDiri} together with Theorem \ref{thm1} provides a proof that polynomials without zeros in $\D=\mathbb{B}_1$ are cyclic in the Dirichlet space $D$. That is known, but it is interesting that the membership in $N^+(D)$ can be established by membership in $D$. Similarly, the following Theorem implies a new proof that for $d=2$ all polynomials without zeros in $\Bd$ are cyclic in $H^2_d$, see \cite{KosinskiVavitsas} or \cite{CyclicWeighedBesov_I}, \Theoonesix.

\begin{theorem}\label{stablePolys} If $\om$ is an admissable measure on $\Bd$ of the type $$d\om(\zeta)= u(r)2rdrd\sigma(z), \zeta=rz,$$ for some $u\in L^\infty[0,1]$, and if $p$ is a stable polynomial, then $$F(z)=\log(1+\log(\frac{2^n p(0)}{p(z)}))\in B^1_\om,$$ where $n\ge $ the degree of $p$.
\end{theorem}
Recall that if $d=2$ and $\om=V=$ Lebesgue measure on $\mathbb{B}_2$, then $H^2_2=B^1_\om$. Hence the hypothesis of Theorem \ref{thm1} for stable polynomials can be verified by membership in $H^2_d$.
\begin{proof} It is clear that we may assume $u(r)=1$ for all $r\in [0,1]$. If $z\in \dB$, then write $F_z$ for the slice function $F_z(\lambda)=F(\lambda z)$. Note that $|RF(\lambda z)|=|\lambda (F_z)'(\lambda|\le |(F_z)'(\lambda)|$. Hence
\begin{align*}\int_{\Bd} |RF(\zeta)|^2 d\om(\zeta)&= \int_{\dB}\int_0^1 \int_0^{2\pi}|RF(e^{it}rz)|^2 \frac{dt}{2\pi}2rdrd\sigma(z)\\
&\le\int_{\dB} \int_\D |(F_z)'(\lambda)|^2\frac{dA(\lambda}{\pi} d\sigma(z)\\
&\le 16n^2
\end{align*}
by Lemma \ref{lem:LogPolyDiri} applied with the slice polynomials $p_z(\lambda)=p(\lambda z)$.
\end{proof}
 In Section 2 of \cite{CyclicWeighedBesov_I} it was shown that for $d\ge 3$ the polynomials $p_1(z)=1-3^{3/2}z_1z_2z_3$ and $p_2(z)=1-(z_1^2+z_2^2+z_3^2)$ are examples of stable polynomials that are cyclic in $H^2_d$ and whose zero set inside $\dB$ is 2-dimensional. The proofs are based on the fact that the following two linear maps of $D$ into $H^2_d$ are bounded and bounded below (see Lemmas \twooneandtwofive of \cite{CyclicWeighedBesov_I}):
$$ V_1f(z)=f(3^{3/2}z_1z_2z_3), V_2f(z)=f(z_1^2+z_2^2+z_3^2).$$
 Thus, the fact that $\log(1+\log\frac{2}{1-z})\in D$ implies that $\log(1+\log\frac{2}{p_j}) \in H^2_d$ for $j=1,2$.
\section{Checking for $N^+$ by checking membership in $\HH$}
We have seen that in some cases one can verify the hypothesis of Theorem \ref{thm1} for a function $f$ with bounded argument and $\|f\|_\infty\le 1$ by showing that $F_n=G_n\circ \log(1/f)\in \HH$ for some $n\in \N$. That raises the following question:

\begin{ques} If $\HH$ is a radially weighted Besov space that is also a complete Pick space, and if $F\in \HH$ with $\mathrm{Re }F\ge 0$, then is $\log (1+F)\in \HH$?
\end{ques}
Thus, if this were true, then $G_n\circ \log(1/f)\in \HH$ would imply $G_{n+1}\circ \log(1/f)\in \HH$.

If $\HH=B^N_\om$ for some $N\in \N$, and if in addition to the hypothesis of the question we have that $\frac{1}{1+F}\in \Mult(B^{N-1}_\om)$, then the answer is yes. That is because $R\log(1+F)=\frac{RF}{1+F}\in B^{N-1}_\om$ since $RF\in B^{N-1}_\om$. Hence $\log(1+F)\in B^N_\om$. That will happen if $N=1$, and it will also happen for any $N$ if $F\in B^N_\om\cap \Mult(B^{N-1}_\om)$, because then the hypothesis along with Theorem \onethree of \cite{CyclicWeighedBesov_I}  implies that $\frac{1}{1+F}\in \Mult(B^{N-1}_\om)$.

Of course, if the $F$ arises in an application of Theorem \ref{thm1}, then $F$ is bounded, if and only if $f$ is bounded away from 0, and then we already know that $f$ is cyclic by \Theooneeight of \cite{CyclicWeighedBesov_I}. For possibly unbounded functions we have only been able to show the following Theorem.

For $\gamma>0$ let $\HH_\gamma$ be the Hilbert function space with reproducing kernel $(1-\la z, w\ra)^{-\gamma}$. Then $\HH_1=H^2_d$ and  $f \in \HH_\gamma$ if and only if $R^m f \in \HH_{\gamma+2m}$ and these are radially weighted Besov spaces, see e.g. \cite{RiSunkesHankel}. Let $\HB$ denote the Bloch space on the ball,
$$\HB=\{f \in \Hol(\Bd): \sup_{z\in \Bd} |Rf(z)|(1-|z|)<\infty\}.$$

\begin{theorem} If  $\gamma>0$  and if $F\in \HH_\gamma \cap \HB$ with Re$F>0$, then $\log(1+F)\in \HH_\gamma$\end{theorem}
\begin{proof}
As on page 2584 of \cite{RiSunkesHankel} we use Faa di Bruno's formula for the nth order derivative of a composition of functions, see \cite{FaadiBrunoRef}.
Let $g(z)= \log (1+z)$ for $z$ in the right half plane, and suppose $\mathrm{Re} F(z) >0$. Then for $k\ge 1$ we have
$$g^{(k)}(z)= \frac{(-1)^{k-1}(k-1)!}{(1+z)^k}.$$
Choose $m\in \N_0$ such that $\gamma+2m> d$, then $\HH_{\gamma+2m}= L^2_a(\om)$ with equivalence of norms, where $d\om=(1-|z|^2)^{\gamma+2m-d-1} dV$, see e.g. \cite{RiSunkesHankel}. Thus, $\log(1+F)\in \HH_\gamma$ if and only if $R^m\log(1+F)\in L^2(\om)$.

Let $A_m$ be the set of all $m$-tuples $\eta=(\eta_1,...,\eta_m)$ of nonnegative integers that satisfy $\sum_{i=1}^M i\eta_i=m$, and write
$$T_\eta(F)= \prod_{i=1}^m (R^iF)^{\eta_i}.$$
Then,  by the Faa di Bruno formula
$$R^m(\log(1+F))= \sum_{\eta\in A_m} \frac{m!}{\eta!} \frac{(-1)^{|\eta|-1}(|\eta|-1)!}{(1+F)^{|\eta|}}\prod_{j=1}^m \left(\frac{1}{j!}\right)^{\eta_j}T_\eta(F).$$

Now Lemma 5.2 of  \cite{RiSunkesHankel}  says that if $F\in \HH_\gamma \cap \HB$, then $T_\eta(F) \in \HH_{\gamma+2m}$ for all $\eta\in A_m$. Since $\mathrm{Re}F>0$ we have that $\frac{1}{1+F}$ is bounded, hence $R^m\log(1+F)\in L^2(\om)$.

\end{proof}
\bibliographystyle{amsplain-nodash} 
\bibliography{cyclicbib2}

\end{document}